\documentclass[11 pt]{amsart}
\usepackage{graphicx} 
\usepackage{amsmath}
\usepackage{mathtools}
\usepackage
{amstext, amscd, setspace, tikz-cd,mathtools, enumerate, graphics, latexsym, siunitx}
\usepackage{pst-node}
\usepackage{tikz-cd} 
\usepackage{amssymb}
\usepackage{hyperref}

\usepackage{PTSerif} 
\usepackage{graphicx}
\usepackage[cmtip,all]{xy}

\newcommand{\R}{{\mathbb R}}

\newcommand{\C}{{\mathbb C}}

\newcommand{\Hom}{{\rm Hom}}
\newcommand{\End}{{\rm End}}

\newcommand{\Tr}{{\rm Trace~}}

\input{xypic}
\xyoption{all}

\newtheorem{cor}[subsubsection]{{\bf Corollary}}
\newtheorem{defn}[subsubsection]{{\bf Definition}}

\newtheorem{lem}[subsubsection]{{\bf Lemma}}
\newtheorem{thm}[subsubsection]{{\bf Theorem}}

\newtheorem{prop}[subsubsection]{{\bf Proposition}}
\newtheorem{rem}[subsubsection]{{\bf Remark}}
\theoremstyle{definition}

\numberwithin{equation}{subsection}

\title[Infinitesimal  $\boldsymbol{ \tau}$-Isospectrality]{On Infinitesimal $\boldsymbol{ \tau}$-Isospectrality of Locally Symmetric Spaces}

\author{Chandrasheel Bhagwat, Kaustabh Mondal and Gunja Sachdeva}

\date{\today}


\address{Indian Institute of Science Education and Research, Dr.\,Homi Bhabha Road, Pashan, Pune 411008,  INDIA.}

\email{cbhagwat@iiserpune.ac.in}

\address{Indian Institute of Science Education and Research, Dr.\,Homi Bhabha Road, Pashan, Pune 411008,  INDIA.}

\email{kaustabh.mondal@students.iiserpune.ac.in}

\address{BITS Pilani, K.K. BIRLA GOA Campus, NH 17B, Bypass Road, Zuarinagar, Sancoale, Goa 403726, India.}

\email{gunjas@goa.bits-pilani.ac.in}

\begin{document}

\begin{abstract}

Let $(\tau, V_{\tau})$ be a finite dimensional representation of a maximal compact subgroup $K$ of a connected non-compact semisimple Lie group $G$, and let  $\Gamma$ be a uniform torsion-free lattice in $G$. We obtain an infinitesimal version of the celebrated Matsushima-Murakami formula, which relates the dimension of the space of automorphic forms associated to $\tau$ and multiplicities of irreducible $\tau^\vee$-spherical spectra in $L^2(\Gamma \backslash G)$. This result gives a promising tool to study the joint spectra of all central operators on the homogenous bundle associated to the locally symmetric space and hence its infinitesimal $\tau$-isospectrality.  Along with this we prove that the almost equality of $\tau$-spherical spectra of two lattices assures the equality of their $\tau$-spherical spectra.

%\tau^\vee-spherical

\end{abstract}

\maketitle
\section{Introduction}
\subsection{$\boldsymbol{\tau}$-Isospectrality and Representation Equivalence}  \label{center-action}Let $G$ be a connected non-compact semisimple Lie group with finite centre. Let $K$ be a maximal compact subgroup in $G$. Then $G/K$ is a symmetric space which carries a $G$-invariant Riemannian metric induced by the Ad($G$)-invariant inner product on the Lie algebra $\mathfrak{g}$ of $G$. For any finite dimensional complex representation $(\tau, V_{\tau})$ of $K$, one has the homogeneous vector bundle 
$E_{\tau}$ on $G/K$ (see Sec. \ref{iso-rep} for the details) whose smooth sections are given by the space  
\[\mathcal A^\infty(G/K, \tau): =  \{ \phi : G \rightarrow V_{\tau} \ | \ \phi(gk)=\tau(k^{-1})(\phi(g))\ \text{for all}\ k \in K, g \in G\}.\]
 
This space $\mathcal A^\infty(G/K, \tau)$ is $\mathfrak{U(g)}^K$-stable. In particular, the centre $\mathfrak{Z(g)}$
of the universal enveloping algebra acts on $\mathcal A^\infty(G/K, \tau)$, and the Casimir element in the centre induces a second order elliptic differential operator $\Delta_{\tau}$ on $\mathcal A^\infty(G/K, \tau)$, which is also induced from the Laplace-Beltrami operator acting on the smooth functions on $G/K$.\smallskip

Let $\Gamma$ be a uniform lattice in $G$. Then $X_{\Gamma}:=\Gamma \backslash G /K$ is a compact locally symmetric space which is manifold if $\Gamma$ acts freely on $G/K$. The space $X_{\Gamma}$ has a Riemannian metric induced from $G/K$.
 We denote by $V_{\Gamma, \tau}$ to be the space of all $\Gamma$-invariant smooth sections in $\mathcal A^\infty(G/K, \tau)$.
 % To avoid complexity in notation we call this V_{\Gamma, \tau}.
 Then $V_{\Gamma, \tau}$ is stable under the action of $\mathfrak{U(g)}^K$, in particular of the centre $\mathfrak{Z(g)}$. The Casimir element $C$, induces the second order self-adjoint elliptic operator $\Delta_{\tau, \Gamma}$ on $V_{\Gamma, \tau}$, which has non-negative discrete spectrum of eigenvalues with finite multiplicities. The set of eigenvalues with multiplicity is denoted by Spec($\Delta_{\tau, \Gamma}$). Such spectrum of a locally symmetric space is closely related to the multiplicity of irreducible representations occurring in the right regular representation $L^2(\Gamma \backslash G)$ of $G$. Here we mention the famous Matsushima-Murakami formula (\cite{MR0266238}, see also Prop. 2.4 of \cite{MR3299293}):

 % \tau should be replaced by tau^\vee %
\begin{equation}\label{1.1.1}
      \text{Mult}_{\Delta_{\tau, \Gamma}}(\lambda) = \sum_{\pi \in \widehat{G}, \pi(C) =\lambda} m(\pi, \Gamma)\ {\rm dim}(\Hom_K(\tau^\vee, \pi|_K)),
     \end{equation} 
$ \text{where}, m(\pi, \Gamma)\ \text{is the multiplicity of}\ 
      \pi \text{ in the right regular representation}\ L^2(\Gamma \backslash G).$
\smallskip

Let $\widehat{G}_{\tau}: = \{ \pi \in \widehat{G} : \Hom_K(\tau, \pi|_K) \neq 0\}$. Two uniform lattices $\Gamma_1$ and $\Gamma_2$ are called \textit{$\tau$-representation equivalent} if $m(\pi, \Gamma_1)=m(\pi, \Gamma_2)$ for all $\pi \in \widehat{G}_{\tau}$. 
 From Eq. \ref{1.1.1}, it is easily observed that if $m(\pi, \Gamma_1)=m(\pi, \Gamma_2)$ for all $\pi \in \widehat{G}_{\tau^\vee}$, then Spec($\Delta_{\tau, \Gamma_1}$) = Spec($\Delta_{\tau, \Gamma_2}$) for $X_{\Gamma_1}$ and $X_{\Gamma_2}$, respectively. Therefore, if two lattices $\Gamma_1$ and $\Gamma_2$ are $\tau^\vee$-representation equivalent, then the corresponding elliptic operators have the same spectrum for the locally symmetric spaces $X_{\Gamma_1}$ and $X_{\Gamma_2}$, respectively i.e. they are \textit{$\tau$-isospectral} (see Prop. 2.5 of \cite{MR3299293}).\smallskip
 
 % \tau^\vee isospectral % \smallskip

 \begin{defn}
Two co-compact lattices $\Gamma_1$ and $\Gamma_2$ are called almost-$\tau$-representation equivalent if $m(\pi, \Gamma_1)$ is equal to $ m(\pi, \Gamma_2)$ for all but finitely many $\pi \in \widehat{G}_{\tau}$. 
\end{defn}
 
 In the first half of this paper, we give an affirmative answer of the question: whether the almost-$\tau$-representation equivalence implies $\tau$-representation equivalence between two uniform torsion free lattices $\Gamma_1$ and $\Gamma_2$ in non-compact symmetric space $G/K$ with arbitrary rank, and for any finite dimensional representation $(\tau, V_{\tau})$ of $K$.

\begin{thm} \label{thm1.3.1}
\textit{Let $G$ be a non-compact connected semisimple Lie group with finite center. If for two uniform torsion free lattices $\Gamma_1$ and $\Gamma_2$, $$m(\pi, \Gamma_1)=m(\pi, \Gamma_2)$$ for all but finitely many $\pi \in \widehat{G}_{\tau}$, then $\Gamma_1$ and $\Gamma_2$ are $\tau$-representation equivalent lattices.} 

\end{thm}

\begin{rem}
Consequently, the above hypothesis implies that $X_{\Gamma_1}$ and $X_{\Gamma_2}$ are $\tau^\vee$-isospectral. When $\tau$ is trivial, this reduces to the results in \cite{MR2836013}. 
In \cite{Bha-Sa}, a special case of Thm.\ \ref{thm1.3.1} is proved for the group ${\rm PSL}(2, \R)$.

\end{rem}
\smallskip

\subsection{Infinitesimal $\boldsymbol{\tau}$-Isospectrality}
For an infinitesimal character $\chi$ of $\mathfrak{Z(g)}$, let $V_{\chi, \Gamma, \tau}:=\{ \phi \in V_{\Gamma, \tau} \ | \ z. \phi = \chi(z) \phi$ for all $z \in \mathfrak{Z(g)}\}$.\smallskip

We define a refinement of the notion of $\tau$-isospectrality as follows.
\smallskip

\begin{defn}
Two locally symmetric spaces $X_{\Gamma_1}$ and $X_{\Gamma_2}$ of non-compact type are \textit{infinitesimally $\tau$-isospectral} if
\[ {\rm dim}\ V_{\chi, \Gamma_1, \tau} = {\rm dim}\ V_{\chi, \Gamma_2, \tau}
\ \text{ for all} \ \chi \in \widehat{\mathfrak{Z(g)}}.\] 
\end{defn}

 Let $[\chi]$ be the set of all irreducible representations of $G$ which has infinitesimal character $\chi$. It is known that $[\chi]$ is a finite subset of $\widehat{G}$ [Cor. 10.37, \cite{Kn}]. We have obtained the following variant of Matsushima-Murakami formula: 

\begin{thm}\label{thm2}
Let $G$ be a connected non-compact semisimple Lie group. Assume that $\Gamma$ be a uniform lattice in $G$. Then for any $\chi \in \widehat{\mathfrak{Z(g)}}$ and for any finite dimensional representation $\tau$ of $K$, 
   \[ {\rm dim}\ V_{\chi, \Gamma, \tau}=\sum_{\pi \in [\chi]} m(\pi, \Gamma)\ {\rm dim}(\Hom_K(\tau^\vee, \pi|_K)). \]
\end{thm}

\begin{rem}
For the rank one semisimple Lie group $G$, an infinitesimal character $\chi \in \widehat{\mathfrak{Z(g)}}$ is completely determined by its value on the Casimir element. Therefore the above result becomes the earlier mentioned Matsushima-Murakami formula Eq. \ref{1.1.1}. 
\end{rem}

\begin{rem}
    Let $\tau=\tau_p$ be the $p$-th exterior power of the adjoint representation of $K$ on $\mathfrak{p}^*_{\mathbb{C}}$. The associated homogeneous vector bundle is identified with the $p$-th exterior product of the co-tangent bundle on $G/K$. In \cite{MR0222908} , Matsushima proved the relationship between the dimension of harmonic $p$-forms on $X_{\Gamma}$ and the multiplicity of irreducible representations in $L^2(\Gamma \backslash G)$ that occur with a nonzero $\tau_p$-isotypic component. Therefore Thm.\ref{thm2} can be seen as an infinitesimal version of Eq. \ref{1.1.1}.
    \end{rem}

\begin{cor}\label{cor2}
 If $\Gamma_1$ and $\Gamma_2$ are $\tau^\vee$-representation equivalent then the spaces $X_{\Gamma_1}$ and $X_{\Gamma_2}$ are infinitesimally $\tau$-isospectral. 
\end{cor}

\begin{rem}
If $\tau$ is trivial, the above formula in Thm.~\ref{thm2} reduces to \cite{MR2836013}.
\end{rem}

\begin{rem}
The finite set $[\chi]$ of all irreducible representations of a real reductive group (informally a $\chi$-packet) is quite close to Langlands L-packet that consists of irreducible admissible representations having same L-parameter. It is known that the irreducible representations from a fixed L-packet have same infinitesimal character i.e., they cannot be distinguished by the spectral data. According to our knowledge, the converse of this is not known in general. It is interesting to formulate a suitable variant of Matsushima-Murakami formula in terms of L-packet instead, and that might reflect some analogous relationship between the multiplicities of irreducible representations in a given L-packet and isospectrality.
\end{rem}

\subsection{Literature Review and Methodology} We briefly review some literature related to the main results in this paper. The question : \textit{Can the space $X_{\Gamma}$ (upto isometry) be determined by its spectrum?} has been of great interest over last few decades. The works of Milnor (see \cite{Mil}) and Vigne\'ras (see \cite{Vigneras1980}) are in the frontline to answer this negatively. Later, the construction of non-isometric isospectral manifolds by Sunada brought significant arithmetic flavour in this context \cite{Sunada1985}. In another important work  \cite{MR2676159} in this context, C.\,S.\,Rajan studied  the appropriate arithmetic properties which are determined by the spectrum. Other than the ambitious inverse spectral problem there is another converse of the above discussion: \textit{Given two $\tau$-isospectral spaces $X_{\Gamma_1}$ and $X_{\Gamma_2}$, are $\Gamma_1$ and $\Gamma_2$ $\tau^\vee$-representation equivalent?} This is called the representation-spectral converse for the pair $(G, K)$. This has gotten attention for quite some time. There is a slightly weaker version of this, defined as follows: 

\begin{defn}
Two locally symmetric spaces $X_{\Gamma_1}$ and $X_{\Gamma_2}$ with the same universal cover $G/K$ are said to be almost-$\tau$-isospectral if $\text{Mult}_{\Delta_{\tau, \Gamma_1}}(\lambda)=\text{Mult}_{\Delta_{\tau, \Gamma_2}}(\lambda)$ for all but finitely many $\lambda$. 
\end{defn}

In the direction of `{\it almost-$\tau$-representation equivalence implies isospectrality}',  several approaches have been made. In \cite{MR4099801}, E. Lauret and R. Miatello
  answered this question in the case of compact Riemannian symmetric space $G/K$, i.e., for compact group $G$. 
 They further proved that the multiplicity of an appropriate finite subset of $\widehat{G}_{\tau}$ determines all multiplicities. Later, they studied the representation-spectral converse in the context of simply connected compact Riemannian symmetric space $G/K$ of rank one and proved there are infinitely many $\tau \in \widehat{K}$ such that almost-$\tau$-isospectrality implies $\tau^\vee$-representation equivalence and hence $\tau$-isospectral \cite{MR4337467}. \smallskip

 For many non-compact homogeneous Riemannian symmetric spaces, Pesce \cite{MR1363805} proved the validity of representation-spectral converse where $\tau$ is trivial. When $\tau$ is trivial the $\tau$-spectra are called spherical spectra. Bhagwat-Rajan\cite{MR2836013} answered that almost-spherical representation equivalence implies spherical representation equivalence. In fact, they described\cite{MR2836013} that for a spherical irreducible representation $\pi \in \widehat{G}_1$, there exists a character $\lambda_{\pi}$ of the algebra of $G$-invariant differential operators on $X_{\Gamma}$ such that $\text{Mult}(\lambda_{\pi})=m(\pi, \Gamma)$ and conversely. Later, Kelmer \cite{MR3256189} obtained several density results which relates the isospectrality and representation equivalence via the notion of length equivalence for the case of $X_{\Gamma}$ with real rank one and of non-compact type (having non-compact universal cover), in particular for compact hyperbolic manifolds.\smallskip

Let us briefly outline our methods for obtaining Thm.$\ref{thm1.3.1}$. We employ the well-established and effective tool, namely the Selberg Trace Formula to study the representation spectra of $L^2(\Gamma \backslash G)$ with respect to a convolution operator. Since our focus is on the $\tau$-spherical irreducible representations in $L^2(\Gamma \backslash G)$, we need to annihilate the non $\tau$-spherical component using appropriate test functions. To achieve this, we instead consider the right regular representation $L^2(\Gamma \backslash G, V_{\tau})$ (see $\ref{2.1.1}$) of $G$. We also utilize an algebra, denoted by $C_c^{\infty}(G, K, V_{\tau})$ consisting of compactly supported smooth $\End({V_{\tau}})$-valued test functions on $G$ which are $\tau$-equivariant (see $\ref{2.1.3}$). One can suitably define the convolution operator on $L^2(\Gamma \backslash G, V_{\tau})$ for such test functions. It is straightforward to compute both the spectral and geometric expansions of the trace of these convolution operators, which leads to the required Selberg Trace Formula for $L^2(\Gamma \backslash G, V_{\tau})$ (see \ref{3.1.8}). The advantage of using such operator valued $\tau$-equivariant test functions is the corresponding convolution operator annihilates the non $\tau$-spherical representation spectra (see \ref{prop4.1.3}). Moreover, due to our hypothesis, taking the difference between the trace formula for $\Gamma_1$ and $\Gamma_2$ yields a finite linear combination of Harish Chandra character distribution associated with finitely many $\tau$-spherical irreducible representations. In Prop. $\ref{prop4.1.6}$, we construct a left $K$-saturated open set in $G$ that avoids all conjugacy classes $[\gamma]_G$ for $\gamma \in \Gamma_1 \cup \Gamma_2$. For the aforementioned test functions supported on this $K$-stable open set, the orbital integrals, and thus the entire geometric side vanish. Finally, the remainder of the argument relies on the analyticity of locally integrable character functions and the linear independence of character distributions for inequivalent irreducible representations. 

%From the above theorems, almost $\tau_p$-representation equivalence implies  $p$-isospectrality and consequently, infinitesimally $p$-isospectrality. 
%In this case, $\Delta_{\tau, \Gamma}$ coincides with the Hodge-Laplace operator acting on $\Gamma$-inavariant smooth $p$-forms. \smallskip

\vspace{8pt}

The organization of the article is as follows: In Sec. [\ref{sec2}], we set up the preliminaries and recall the Harish Chandra character distributions, Isospectrality and Representation equivalence. In Sec. [\ref{sec3}], we explicitly calculate the Selberg Trace Formula for $L^2(\Gamma \backslash G), V_{\tau})$. In Sec. [\ref{sec4}], we discuss the required lemmas and propositions to prove Thm.[\ref{thm1.3.1}]. In Sec. [\ref{sec5}], we provide some observations and complete the proof of the Thm. \ref{thm2}. In Sec.\ [\ref{DS}], we consider the case where the group $G$ has discrete series representations and using the results from \cite{MR1104438} and Thm.~\ref{thm2}, we show that the dimension of $\chi$-eigenspace of the automorphic forms of type $\tau^\vee$ is equal to the $q_{\lambda}$-th $L^2$-cohomology of the automorphic line bundle $\Gamma \backslash \mathcal{L}_{\lambda}$ associated to the discrete series representation $\pi$ with minimal $K$-type $\tau$ and infinitesimal character $\chi$.

\smallskip

{\bf ACKNOWLEDGEMENTS:} K.\, Mondal thanks the program Zariski dense subgroups, number theory and geometric application (\url{ICTS/zdsg2024/01}) at International Centre for Theoretical Sciences(ICTS), Bengaluru, India in January, 2024. He also sincerely thanks Prof.\, C.\, S.\, Rajan for suggesting him to write the $\tau$-equivariant Trace Formula during this workshop. He also thanks Prime Minister Research Fellowship (PMRF) Govt. of India for supporting this work partially. G. Sachdeva's research is supported by Department of Science Technology-Science and Engineering Research Board, Govt. of India POWER Grant [SPG/2022/001738].\smallskip

\section{Preliminaries}\label{sec2}

\subsection{Basic setup} Let $G$ be a connected non-compact semisimple Lie group and let $K$ be a maximal compact subgroup of $G$ with Lie algebras $\mathfrak{g}$ and $\mathfrak{k}$, respectively. Then the homogeneous space $G /K$ is a symmetric space with a $G$-invariant metric. Let $(\tau, V_{\tau})$ be a finite dimensional representation of $K$. Let $\Gamma$ be a uniform torsion-free lattice in $G$.\smallskip

 We consider the right regular representation $\rho$ of $G$ on the space 

\begin{equation}\label{2.1.1}
\begin{split}L^2(\Gamma \backslash G, V_{\tau}) & =\{ \phi : G \rightarrow V_{\tau}\ | \ \phi \text{ is measurable}, \\ & \phi(\gamma g)  =\phi(g) \text{ for all }   \gamma \in \Gamma,
 \int \limits_{\Gamma \backslash G} |\phi(g)|^2\ dg  < \infty\}.
\end{split}
\end{equation}

 Since $\Gamma \backslash G$ is compact, the right regular representation $L^2(\Gamma \backslash G)$ of $G$ is completely reducible and each irreducible subrepresentation occurs with finite multiplicity. In other words, $$L^2(\Gamma \backslash G) \cong \widehat{\bigoplus_{\pi \in \widehat{G}}} \ m(\pi, \Gamma)\ W_{\pi}$$ with $m(\pi, \Gamma) < \infty$ for all $\pi \in \widehat{G}$.
\smallskip

We choose an orthonormal basis for each copy of $\pi$, taking union of those we get an orthonormal basis $\mathcal{W}$ of $L^2(\Gamma \backslash G)$. Let $\{v_i\}_{i=1}^{n}$ be an orthonormal basis of $V_{\tau}$ \label{orthonormalbasis}. For any $\psi \in \mathcal{W}$, we define
 $\psi_i(x):=\psi(x)v_i$. 
 Then $\{ \psi_i \ | \ \psi \in \mathcal{W} \text{ for all } 1 \leq i \leq n\}$ forms an orthonormal basis of $L^2(\Gamma \backslash G, V_{\tau})$. Note that 
 $L^2(\Gamma \backslash G, V_{\tau}) = L^2(\Gamma \backslash G) \otimes V_{\tau}$.\smallskip

 Therefore, we get
\begin{equation}\label{2.1.2}L^2(\Gamma \backslash G, V_{\tau}) \cong 
\widehat{\bigoplus\limits_{\pi \in \widehat{G}}} \ m(\pi, \Gamma)\ (W_{\pi} \otimes V_{\tau} ).\end{equation}

For each copy of $W_{\pi} \otimes V_{\tau}$, we identify a subspace $V_{\pi}$ in $L^2(\Gamma \backslash G, V_{\tau})$. The subspace $V_{\pi}$ is not an irreducible $G$-subspace, rather it is direct sum of dim $V_\tau$ many copies of $W_{\pi}$.\smallskip

We introduce a space $C_c^{\infty}(G, K, V_{\tau})$ defined by 
\begin{equation}\label{2.1.3}
    \begin{split}
\{ f : G \rightarrow \End(V_{\tau}) \ | \ & f \text{ is compactly supported and smooth such that }\\
& f(kx)=f(x)\tau(k^{-1})  \text{ for all }  x \in G, k \in K\}. 
    \end{split}
\end{equation} 

For $f \in C_c^{\infty}(G, K, V_{\tau})$, let $h_f(x):=\Tr(f(x))$ for all $x \in G$. Now for any $f \in C_c^{\infty}(G, K, V_{\tau})$, we define the convolution operator $\rho(f) :L^2(\Gamma \backslash G, V_{\tau} ) \rightarrow L^2(\Gamma \backslash G, V_{\tau})$ by  \begin{equation}\label{2.1.4}
\rho(f) \phi (x) = \int \limits_G f(y) (\phi(gy)) \ dy \quad \text{for all}\ 
\phi \in L^2(\Gamma \backslash G, V_{\tau} ).
\end{equation}

\begin{prop}\label{prop2.1.1}
 \textit{Each $V_{\pi}$ is $\rho(f)$-stable subspace for all $f \in C_c^{\infty}(G, K, V_{\tau})$.}\smallskip
\end{prop}
\begin{proof}
    Let $\mathcal{U} \subset \mathcal{W}$ be the subset such that $\mathcal{U}$ is an orthornormal basis of $W_{\pi}$. Then $\{ \phi_i : \phi \in \mathcal{U}, 1 \leq i \leq n\}$\ is an orthornormal basis of $V_{\pi}$. Thus it suffices to show that 
$$ \langle \rho(f) \phi_i, \psi_j \rangle =0 \quad \text{for any}\ \phi \in \mathcal{U}, \psi \in \mathcal{W} \setminus \mathcal{U} \ \text{and any}\ 1 \leq i, j \leq n. $$

We compute,
\begin{equation}
\begin{split}
\langle \rho(f) \phi_{i} , \psi_{j} \rangle & = \int \limits_{\Gamma \backslash G} \langle \rho(f) \phi_{i}(g), \psi_{j}(g) \rangle\  dg\\
& = \int \limits_{\Gamma \backslash G} \langle \int\limits_G \phi(gy) f(y) v_i, \psi(g) v_j \rangle\ dy\ dg\\
& = \int \limits_{\Gamma \backslash G} \int\limits_G \phi(gy) \overline{\psi(g)} \langle f(y)v_i, v_j \rangle\ dy\ dg\\
& = \int \limits_G \langle f(y)v_i, v_j \rangle \int\limits_{\Gamma \backslash G} \phi(gy)\overline{\psi(g)}\ dg\ dy\\
& = 0.
\end{split}
\end{equation}
\end{proof}

\subsection{Results on Harish-Chandra characters} We will be using the following two important results on \textit{Harish-Chandra characters}: \smallskip

\noindent \textbf{Harish-Chandra character distribution:}  Let $(\pi, W_{\pi})$ be an irreducible unitary representation of $G$. Let $C_c^{\infty}(G)$ be the space of all compactly supported smooth functions on $G$. For any $f \in C_c^{\infty}(G)$, the convolution operator $\pi(f)$ on $W_{\pi}$ is defined by $\pi(f)v=\int \limits_G f(g) \pi(g)v \ dg$ for all $v \in W_{\pi}$. It is a trace class operator [Theorem 10.2, \cite{Kn}] and 
$$ \chi_{\pi}(f)=\Tr(\pi(f)) \  \text{ for all } f \in C_c^{\infty}(G).$$

\begin{thm}\label{thm2.2.1}
\textit{Let $\{\pi_i\}$ be a finite collection of mutually inequivalent unitary irreducible representations of $G$. Then their characters $\{\chi_{\pi_i}\}$ are linearly independent distributions on $G$.}
\end{thm}

\begin{proof}
    Reader is referred to [Theorem 10.6, \cite{Kn}].
\end{proof}

\begin{thm}\label{thm2.2.2}
\textit{Let $\pi$ be an irreducible unitary representation of $G$. Then the distribution character $\chi_{\pi}$ is given by a locally integrable function $\phi_{\pi}$ on $G$ i.e. for $f \in C_c^{\infty}(G)$, $$\chi_{\pi}(f)=\int \limits_G f(g) \phi_{\pi}(g) \ dg.$$}

\textit{Moreover, the restriction of $\phi_{\pi}$ to the regular set of $G$ is a real analytic function invariant under conjugation. }
\end{thm}

\begin{proof}
    Reader is referred to [Theorem 10.25, \cite{Kn}].
\end{proof}
\smallskip

\subsection{Isospectrality and representation equivalence}\label{iso-rep}

 There is a natural homogeneous vector bundle $E_{\tau}$ on $G/K$ that is associated with the representation $(\tau , V_{\tau})$ of $K$. The space of all smooth global sections of $E_{\tau}$ can be realised as $\mathcal A^\infty(G/K, \tau) =\{ \phi : G \rightarrow V_{\tau}\ |\ \phi$ is smooth, $\phi(xk)=\tau(k^{-1})(\phi(x)) \text{ for all } x \in G, k \in K\}$. Note that $\mathcal A^\infty(G/K, \tau)$ is a $\mathfrak{U(g)}^K$ module. In particular, $\mathfrak{Z(g)} \subset \mathfrak{U(g)}^K$ acts on $\mathcal A^\infty(G/K, \tau)$.\smallskip

 Let $\mathfrak{g}= \mathfrak{k} \oplus \mathfrak{p}$ be the Cartan decomposition of the Lie algebra $\mathfrak{g}$. We choose a basis $\{X_i\}$ and $\{Y_j\}$ of $\mathfrak{k}$ and $\mathfrak{p}$, respectively with respect to a bilinear form $B$ on $\mathfrak{g}$ induced from the killing form such that $B(X_k, X_l)=-\delta_{kl}$ and $B(Y_m, Y_n)=\delta_{mn}$. Let $C$ be the Casimir element given by $C=-\sum X_i^2 + \sum Y_j^2$. The Casimir element $C$ induces a second order symmetric elliptic operator $\Delta_{\tau}$ on $\mathcal A^\infty(G/K, \tau)$.\smallskip

 Let $\Gamma$ be a uniform torsion-free lattice in $G$, and let 
 $ X_{\Gamma}:= \Gamma \backslash G / K$ be the associated compact locally symmetric space of non-compact type. We consider the vector bundle $E_{\tau, \Gamma}$ on $X_{\Gamma}$ defined by the relation $[\gamma g , w] \sim [g, w]$ for all $\gamma \in \Gamma \text{ and } [g, w] \in E_{\tau}$. The space of all smooth global sections of $E_{\tau, \Gamma}$ can be realised as the space 
 $V_{\Gamma, \tau}=\{ \phi \in \mathcal A^\infty(G/K, \tau) \ | \ \phi(\gamma x)=\phi(x) \text{ for all } \gamma \in \Gamma\}$.
 % For simplicity we denote it by V_{\Gamma, \tau}.
 Hence, the centre $\mathfrak{Z(g)}$ acts on $V_{\Gamma, \tau}$. \smallskip

 Let $\Delta_{\tau, \Gamma}:=\Delta_{\tau}|_{V_{\Gamma, \tau}}.$ This is again a second order symmetric elliptic operator on $X_{\Gamma}$. Its spectrum Spec$(\Delta_{\tau, \Gamma})$ is a discrete subset of the non-negative real numbers. 
 Let $\text{Mult}_{\Delta_{\tau, \Gamma}}(\lambda):=$ {\it the multiplicity of} $\lambda \in$ Spec$(\Delta_{\tau, \Gamma})$. Recall Eq.~\eqref{1.1.1} 
 $$\text{Mult}_{\Delta_{\tau, \Gamma}}(\lambda)  = \sum_{\pi \in \widehat{G}, \pi(C)=\lambda} m(\pi, \Gamma)\ {\rm dim}(\Hom_K(\tau^\vee, \pi|_K)),$$
where $\pi(C)$ is the scalar by which Casimir element $C$ acts on $W_{\pi}$.

\begin{defn}
$ X_{\Gamma_1}$ and $X_{\Gamma_2}$ are called \textit{$\tau$- isospectral} if the operators $\Delta_{\tau, \Gamma_1}$ and $\Delta_{\tau, \Gamma_2}$ have same spectrum (with multiplicity). 
\end{defn}

We denote $\widehat{G}_{\tau}=\{ \pi \in \widehat{G} : \Hom_K(\tau, \pi) \neq 0\}$.

\begin{defn}
$\Gamma_1$ and $\Gamma_2$ are called \textit{$\tau$-representation equivalent} if $m(\pi, \Gamma_1)=m(\pi, \Gamma_2)$ for all $\pi \in \widehat{G}_{\tau}$.
\end{defn}

It is clear from Eq.(~\ref{1.1.1}) if $\Gamma_1$ and $\Gamma_2$ are $\tau^\vee$-representation equivalent, then $X_{\Gamma_1}$ and $X_{\Gamma_2}$ are $\tau$-isospectral.
\smallskip
% \tau^\vee isospectral% 
\smallskip

\subsection{A further refinement}\label{refinement}
 We have seen that $\mathfrak{Z}(\mathfrak{g})$ acts on $V_{\Gamma, \tau}$ (see \ref{center-action}). For any character $\chi \in \widehat{\mathfrak{Z}(\mathfrak{g})}$, we have the $\chi$-eigenspace defined as
 $$V_{\chi, \Gamma, \tau}:=\{ \phi \in V_{\Gamma, \tau}\ | \ z. \phi =\chi(z) \phi \text{ for all }  z \in \mathfrak{Z}(\mathfrak{g})\}.$$

In fact, $V_{\Gamma, \tau}$ decomposes as  $V_{\Gamma, \tau}=\bigoplus \limits_{\chi \in \widehat{\mathfrak{Z(g)}}} V_{\chi, \Gamma, \tau}$.

\vspace{8pt}

We here introduce a notion of two locally symmetric spaces
being {\it infinitesimally $\tau$-isospectral}, defined as follows:  

\begin{defn} Two locally symmetric spaces $X_{\Gamma_1}$ and $X_{\Gamma_2}$ are infinitesimally $\tau$-isospectral if for all characters $\chi \in \widehat{\mathfrak{Z(g)}}$,  
\[ \rm{dim}\ V_{\chi, \Gamma_1, \tau}=\rm{dim}\ V_{\chi, \Gamma_2, \tau}.\]
\end{defn}

\begin{rem}
If $G$ is a real rank one Lie group, the centre $\mathfrak Z(\mathfrak g)$ is the polynomial algebra over $\C$ in the Casimir element $C$. Therefore, the $\chi$-eigenspace is simply the $\lambda$-eigenspace where $\lambda$ is an eigenvalue of the Laplace-Beltrami operator $\Delta_{\tau, \Gamma}$ on $E_{\tau, \Gamma}$. In this case, infinitesimal $\tau$-isospectrality reduces to $\tau$-isospectrality as defined above.

\end{rem}
As in \cite{MR2836013} and \cite{Bha-Sa}, we can expect that the $\text{dim} \ V_{\chi, \Gamma, \tau}$ is related with the multiplicities $m(\pi, \Gamma)$ of $\pi \in \widehat{G}_{\tau^\vee}$ occurring in $L^2(\Gamma \backslash G)$ with infinitesimal character $\chi$. This is the content of Thm.~\ref{thm2}.

\section{Selberg trace formula for $L^2(\Gamma \backslash G, V_{\tau})$}\label{sec3}

Recall $\{\psi_{j} : \psi \in \mathcal{W}, j\in {1,...,n}\}$ is an orthonormal basis of $L^2(\Gamma \backslash G, V_{\tau})$ (see \ref{2.1.1}). For any $f \in C_c^{\infty}(G, K, V_{\tau})$ (see \ref{2.1.3}), let $[f_{ij}]_{n \times n}$ be the matrix representation of $f$ with respect to the basis $\{v_i\}_{i=1}^n$ of $V_{\tau}$. For $f \in C_c^{\infty}(G, K, V_{\tau})$, observe that

\begin{equation}
\begin{split}
\rho(f) \psi_j(x) & = \int \limits_G f (x^{-1} y ) (\psi_{j}(y))\ dy \\
 & = \int \limits_G f ( x ^{-1} y ) ( \psi(y) v _j)\ dy \\
% & = \int_G \psi(y) f(x ^{-1} y)(v_j) dy \\
% & = \int_G \psi(y) ( \sum_{i} f_{ij}(x ^{-1} y) v_i)\ dy \\
 & = \sum_{i} \int \limits_G \psi (y)f _{ij}( x ^{-1} y) v_i\ dy.
\end{split}
\end{equation}

% remove lines 3, 4

Then, we have 
\begin{equation}
\begin{split}
\Tr (\rho(f)) & = \sum_{\psi \in \mathcal{W}} \sum_{j} \langle \rho(f) \psi_{j}, \psi_{j}\rangle \\
 & = \sum_{\psi \in \mathcal{W}} \sum_{j} \int \limits_{\Gamma \backslash G} \langle \sum_{i} \int \limits_G \psi(y) f_{ij}(x ^{-1} y) v_i\ dy , \psi(x) v_j \rangle\ dx \\
% & = \sum_{\psi \in \mathcal{W}} \sum_{j} \sum_{i} \int_{\Gamma \backslash G } \int_{G} \psi(y) f_{ij} ( x^ {-1} y) \overline{\psi(x)} \langle v_i, v_j\rangle dy dx \\
% & = \sum_{i} \sum_{\psi \in \mathcal{W}} \int_{\Gamma \backslash G} \int_{G} \psi(y) f_{ii}(x ^{-1} y) \overline{\psi(x)} \dy\ dx \\
% & = \sum_{i} \sum_{\psi \in \mathcal{W}} \int_{\Gamma \backslash G} \int_{\Gamma \backslash G} \sum_{\gamma \in \Gamma } f_{ii} ( x^{-1} \gamma y) \psi(\gamma y) \overline{\psi(x)} dy dx \\
 & = \sum_{\psi \in \mathcal{W}} \ \int \limits_{\Gamma \backslash G \times \Gamma \backslash G} \sum_{\gamma \in \Gamma} \Tr (f( x^{-1} \gamma y)) \overline{\psi(x)} \psi(y) \ dy\ dx \\
 & = \sum_{\psi \in \mathcal{W}} \ \int \limits_{\Gamma \backslash G \times \Gamma \backslash G} \sum_{\gamma \in \Gamma} h_f( x^{-1} \gamma y) \overline{\psi(x)} \psi(y)\ dy \ dx.
 \end{split}
\end{equation}

Put $K_f(x,y): =\sum\limits_{\gamma \in \Gamma} h_f(x^{-1}\gamma y)$.   
Above equals,
\begin{equation}
\begin{split}
\sum_{\psi \in \mathcal{W}} \ \int\limits_{\Gamma \backslash G \times \Gamma \backslash G} K_f(x,y) \overline{\psi (x)} \psi(y) \ dy \ dx. \\
\end{split}
\end{equation}

Recall that if $T$ is an integral operator on $L^2 (\Gamma \backslash G)$ defined by 
\begin{equation}
\begin{split}
T \psi(x) & = \int\limits_{\Gamma \backslash G} K_f(x,y) \psi(y)\ dy. 
\end{split}
\end{equation}

 Then $T$ is a trace class operator and its trace is given by
\begin{equation}
\begin{split}
\Tr(T) =
%= \sum_{\psi \in \mathcal{W}} \langle T\psi,\psi \rangle \\
%& = \sum_{\psi \in \mathcal{W}} \int_{\Gamma \backslash G} T \psi(x) \overline{\psi(x)} dx \\
%& = \sum_{\psi \in \mathcal{W}} \int_{\Gamma \backslash G} \int_{\Gamma \backslash G} K_f(x,y) \psi(y) dy \overline{\psi(x)} dx \\
 \sum_{\psi \in \mathcal{W}} \ \int\limits_{\Gamma \backslash G \times \Gamma \backslash G} K_f(x, y) \overline{\psi(x)} \psi(y)\ dy\ dx. \\
\end{split}
\end{equation}

On the other hand the trace of such an operator $T$ equals $\int \limits_{\Gamma \backslash G} K_f(x,x)\ dx$.
Therefore, $$\Tr(\rho(f))= \int\limits_{\Gamma \backslash G} K_f(x,x)\ dx=\int\limits_{\Gamma \backslash G} \sum\limits_{\gamma \in \Gamma } h_f( x^{-1} \gamma x)\ dx.$$

From Prop.~\ref{prop2.1.1}, we have
$$\Tr(\rho(f))=\sum\limits_{\pi \in \widehat{G}} m(\pi , \Gamma)\ \Tr(\rho(f)|_{V_{\pi}}).
$$

Now,
\begin{equation}
\begin{split}
\Tr(\rho(f)|_{V_{\pi}}) & = \sum_{j} \sum_{\psi \in \mathcal{W} \cap W_{\pi}} \langle \rho(f) \psi_j , \psi_j \rangle \\
%& = \sum_j \sum_{\psi \in \mathcal{W} \cap W_{\pi}} \int_{\Gamma \backslash G} \langle \rho(f) \psi_j(x) , \psi_j(x) \rangle dx\\
%& = \sum_j \sum_{\psi \in \mathcal{W} \cap W_{\pi}} \int_{\Gamma \backslash G} \langle \int_G \psi(xy)f(y)v_j , \psi(x) v_j \rangle dy dx\\
%& = \sum_j \sum_{\psi \in \mathcal{W} \cap W_{\pi}} \int_{\Gamma \backslash G} \int_G \psi(xy) \overline{\psi(x)} \langle f(y)v_j, v_j \rangle dy dx\\
& = \sum_{\psi \in \mathcal{W} \cap W_{\pi}} \ \int\limits_{\Gamma \backslash G} \int\limits_G \psi(xy) \ \overline{\psi(x)}\ \Tr(f(y))\ dy\ dx\\
& = \Tr(\pi(h_f)).
\end{split}
\end{equation}
\smallskip

Therefore, we have the following \textit{Selberg Trace formula}:

 \begin{equation}\label{3.1.8}
 \begin{split}
 \sum_{\pi \in \widehat{G}} m( \pi, \Gamma)\  \Tr( \pi(h_f)) & =\int\limits_{\Gamma \backslash G} \sum_{\gamma \in \Gamma} h_f(x^{-1} \gamma x)\ dx \\
 & = \sum_{[\gamma] \in [\Gamma]_G} \text{vol}\ (\Gamma_{\gamma} \backslash G_{\gamma}) \int\limits_{G_{\gamma} \backslash G} h_f(x^{-1} \gamma x)\ dx.
 \end{split}
 \end{equation}

\noindent We denote $\int\limits_{G_{\gamma} \backslash G} h_f(x^{-1} \gamma x)\ dx\ \text{by}\ O_\gamma(h_f)$.

\begin{rem}
    This trace formula is a generalisation of the well-known  Selberg trace formula. For more explicit description about the `geometric side', see \cite{Wal}.
\end{rem}

\section{Proof of the first main result Thm.~\ref{thm1.3.1}}\label{sec4}

\subsection{Some preliminary results} We will describe the lemmas and propositions required to prove Thm.~\ref{thm1.3.1}.

\begin{lem}\label{lem4.1.1}
\textit{Let $Kx$ be a left coset of $K$ in $G$ for some $x \in G$. Let $U$ be an open set containing the left coset $Kx$. Then there exists $f \in C_c^{\infty}(G)$ with Supp\ $(f)\subset U$ such that $\int\limits_K f(k x) \chi_{\tau}(k) \ dk \neq 0$. (Here $\chi_{\tau}$ is the character function of $\tau$.)}
\end{lem}

\begin{proof}
For any $f \in C_c^{\infty}(G)$ with ${\rm Supp}(f) \subset U$, we define $\psi \in C^{\infty}(K)$ by $\psi(k):=f(kx)$. Conversely, for any $\psi \in C^{\infty}(K)$, define $f(kx):=\psi(k)$, and extend $f$ smoothly to $U$.\smallskip

Now, if $\int\limits_K f(kx) \chi_{\tau}(k)\ dk =0$ for all $f \in C_c^{\infty}(G)$ such that ${\rm Supp}(f) \subset U$, then 
\[
\int\limits_K \psi(k)\ \chi_{\tau}(k)\ dk =0\  \text{for all}\ \psi \in C^{\infty}(K).\] 
But that implies $\chi_{\tau}$ is identically zero leading to a contradiction.
\end{proof}

\begin{lem}\label{lem4.1.2}
 \textit{Let $Kx \neq Ky$ be two distinct left cosets of $K$ in $G$. Then there exists $F \in C_c^{\infty}(G, K, V_{\tau})$ such that $h_{F}(x)\neq 0$ and $h_{F}|_{Ky}=0$.}
\end{lem}

\begin{proof}
    $Kx \neq Ky$ are disjoint compact sets in $G$. Choose disjoint open sets $U_1, U_2$ with $Kx \subset U_1,Ky \subset U_2$, and $U_1 \cap U_2 = \emptyset$. From the above Lem.~\ref{lem4.1.1}, there exists $f \in C_c^{\infty}(G)$ with ${\rm Supp}\ (f) \subset U_1$ such that $\int\limits_K f(kx) \chi_{\tau}(k)\ dk \neq 0$.

Define, $F(g) : = \int\limits_K f(kg) \tau(k)\ dk$. Then, $F \in C_c^{\infty}(G,K, V_{\tau})$. 

Furthermore, $h_{F}(g)=\int\limits_K f(kg) \chi_{\tau}(k)\ dk$. Hence $h_F(x) \neq 0$ and $h_F(y)=0$ and hence $h_{F}|_{Ky}=0$.
\end{proof}

\begin{prop}\label{prop4.1.3}
\textit{Let $\pi$ be an irreducible representation of $G$ occurring as a subrepresentation of $\rho$. Assume that $\pi$ is not $\tau$-spherical i.e. $\Hom_K(V_{\tau}, V_{\pi})=0$. Then $\rho(f)$ is zero on $V_{\pi}$.}
\end{prop}

\begin{proof}
    It is enough to show that $\rho(f)$ is zero at each element of the orthornormal basis $\{\psi_j : \psi \in \mathcal{U}, 1 \leq j \leq n\}$. Here, $\mathcal{U}=\mathcal{W} \cap W_{\pi}$. Recall that (see Eq.~\eqref{2.1.4}) 
    $$\rho(f)\psi_j(g)= \int\limits_G \psi(gy) f(y)(v_j) \ dy.$$
  
  For a fixed $\psi \in \mathcal{U}$, consider the subspace $U_{\psi}:=\text{span}\ \{\psi_j  : j \in \{1,\ldots, n\}\} $. If $v = \sum\limits_{j=1}^{n} a_j v_j$, we write $\psi_v:=\sum_{j=1}^{n} a_j \psi_j$. Then the subspace $U_{\psi}$ is same as $\{ \psi_v : v \in V_{\tau}\}$. There is an action of $K$ on $U_{\psi}$ by $\tau(k)\psi_v=\psi_{\tau(k)v}.$\
  
  Clearly, $U_{\psi}$ is a representation of $K$ isomorphic to $\tau$. Note that $\widehat{\bigoplus\limits_{\psi \in \mathcal{W}}} U_{\psi} = V_{\pi}$. We show that $\rho(f)|_{U_{\psi}}=0$. In fact, we show that $\rho(f) \in \Hom_K(U_{\psi}, V_{\pi})$. 
  
  Let $\psi_v \in U_{\psi}$ and $k_0 \in K$. Then for every $g \in G$, we have
  
\begin{equation}
\begin{split} 
(\rho(k_0) \rho(f) \psi_v )(g) 
& = (\rho(f) \psi_v)(gk_0)\\
& = \int\limits_G f(y)\ (\psi(gk_0y)v)\ dy\\
& = \int\limits_G \psi(gy)\ f(k_0^{-1} y) v\ dy\\
& = \int\limits_G \psi(gy)\ f(y) \tau(k_0)v\ dy\\
& = \rho(f)\ \tau(k_0)\ \psi_v(g).
\end{split}
\end{equation}

Therefore, $\rho(f)$ is zero on $U_{\psi}$ for all $\psi \in \mathcal{W}$. Consequently, $\rho(f)$ is zero on $V_{\pi}$. 

\end{proof}

\begin{lem}\label{lem4.1.4}
 \textit{Let $\Gamma$ be a torsion-free uniform lattice in $G$. For a non-trivial element $\gamma \in \Gamma$, the conjugacy class $[\gamma]_G$ of $\gamma$ in $G$ is disjoint from $K$.}
\end{lem}

\begin{proof}
    Reader is referred to \cite{MR2836013}.
\end{proof}

\begin{lem}\label{lem4.1.5} 
\textit{If $\gamma \neq e$, then $e \notin K [\gamma]_G$.}
\end{lem}

\begin{proof}
    If there is $k \in K$ and $ x \in G$ such that $k x^{-1} \gamma x =e$, then $x^{-1} \gamma x \in K$; which contradicts the previous lemma.
    
\end{proof}

\begin{prop}\label{prop4.1.6}
\textit{There exists an open set $B$ in $G$ such that $[\gamma]_G \cap B$ is empty for all $\gamma \in \Gamma_1 \cup \Gamma_2$, and $B$ is stable under left $K$ action on $G$.}
\end{prop}

\begin{proof}
    Let $U'$ be a relatively compact open set containing $e$ in $G$. Let $U=KU'$. Then $U$ is relatively compact and therefore it intersects at the most finitely many conjugacy classes $[\gamma]_G$. Since, the natural map $G \rightarrow K \backslash G$ is proper, $K [\gamma]_G$ is closed in $G$. Since $U$ is $K$-stable, $K[\gamma]_G \cap U \neq \emptyset$ if and only if $[\gamma]_G \cap U \neq \emptyset$. Let $E : = \bigcup \limits_{\gamma \neq e} (K [\gamma]_G \cap U)$. Since $E$ is a finite union of closed sets, it is closed and $K$-stable subset of $U$. From Lem.~\ref{lem4.1.5}, we conclude that $e \notin E$. Choose an open set $V$ containing $e$ such that $E \cap V =\emptyset$. Now, let $B=KV \cap K^c$. Then $B$ is the desired open set in $G$.
\end{proof} 
\medskip

\subsection{Proof of Theorem \ref{thm1.3.1}}
We have two uniform torsion free lattices $\Gamma_1$ and $\Gamma_2$ in $G$. So by Eq. (\ref{2.1.2}) $$  (\rho_{\Gamma_i}, L^2(\Gamma_i \backslash G, V_{\tau})) = \widehat{\bigoplus_{\pi \in \widehat{G}}} \ m(\pi, \Gamma_i) V_{\pi} $$ for $ i =1, 2$. Let $t_{\pi}= m(\pi, \Gamma_1)-m(\pi, \Gamma_2)$. By hypothesis there exists a finite set $\mathcal{S} \subset \widehat{G}_{\tau}$ such that $t_{\pi}=0$ for all $\pi \in \widehat{G}_{\tau} \backslash \mathcal{S}$. For $f \in C^{\infty}_c(G, K, V_{\tau})$, $\rho_{\Gamma_i}(f)$ is zero on $V_{\pi}$ if $\pi \notin \widehat{G}_{\tau}$ by the proposition \ref{prop4.1.3}. Let $h_f(y)=\Tr(f(y))$. Therefore from the Selberg trace formula \ref{3.1.8}, we have $$ \sum_{\pi \in \mathcal{S}} t_{\pi} \Tr(\pi(h_f)) = \sum_{[\gamma] \in [\Gamma_1]_G \cup [\Gamma_2]_G} (a( \gamma, \Gamma_1)-a(\gamma, \Gamma_2)) O_{\gamma}(h_f);$$ 

$$ \sum_{\pi \in \mathcal{S}} t_{\pi} \int\limits_G h_f(y) \phi_{\pi}(y)\ dy = \sum_{[\gamma] \in [\Gamma_1]_G \cup [\Gamma_2]_G} (a( \gamma, \Gamma_1)-a(\gamma, \Gamma_2)) O_{\gamma}(h_f).$$

Let $\phi=\sum\limits_{\pi \in \mathcal{S}} t_{\pi} \phi_{\pi}$. Then $$ \int\limits_G h_f(y) \phi(y)  \ dy = \sum_{[\gamma] \in [\Gamma_1]_G \cup [\Gamma_2]_G} (a(\gamma, \Gamma_1)-a(\gamma, \Gamma_2)) O_{\gamma}(h_f).$$ 

Let B be the open set from the proposition \ref{prop4.1.6}. For any $f \in C^{\infty}_c(G, K, V_{\tau})$ supported in B, the orbital integrals on the right side is zero. For such functions $f$, we have  $$ \int\limits_B h_f(y) \phi(y)\ dy =0.$$ 

From Lem.~\ref{lem4.1.2}, the functions $h_f$ separates points in $B$. Hence $\phi$ must vanish on the open subset $B$ of $G$. Since $\phi$ is real analytic (see \ref{thm2.2.2}), it vanishes on all of $G$. By the linear independence of functions $\phi_{\pi}$ (see \ref{thm2.2.1}), we conclude that $m(\pi, \Gamma_1)=m(\pi, \Gamma_2)$ for all $\pi \in \widehat{G}_{\tau}$.

\section{Proof of second main result Thm.~\ref{thm2} }\label{sec5}

\subsection{Some observations} In this subsection, we make some observations that will be useful in the proof of Thm.~\ref{thm2}.\smallskip

Recall (see Sec.\ \ref{iso-rep}) that 
$\mathcal A^\infty(G/K, \tau)$ is the space of smooth sections of the vector bundle $E_\tau$ and $V_{\Gamma, \tau}$ is the subspace defined by 
%Recall (see Sec.\ \ref{iso-rep}) that = \{ \phi : G \rightarrow V_{\tau}\ |\ \phi(xk)=\tau(k^{-1})(\phi(x))$ for all $x \in G, k \in K\}$, . \smallskip 
%The space 
% is defined by 
\[V_{\Gamma, \tau} = \{ \phi \in \mathcal A^\infty(G/K, \tau) \ | \ \phi(\gamma x)=\phi(x)\ \text{ for all}\  \gamma \in \Gamma, x \in G\}.
\]

Let $n=\text{dim}(V_{\tau})$. Recall the choice of orthornormal basis $\{v_i\}_{i=1}^{n}$ of $V_{\tau}$ from Sec.\ref{orthonormalbasis}. For $\phi \in V_{\Gamma, \tau}$, write $\phi(x)=\sum\limits_{i=1}^{n} \phi_i(x) v_i$ for all $x \in G$, where each $\phi_i$ is smooth complex valued function on $\Gamma \backslash G$. \noindent  Recall that $V_{\chi, \Gamma, \tau}$ is the $\chi$-eigenspace of $V_{\Gamma, \tau}$ with respect to the action of $\mathfrak{Z(g)}$. For any $X \in \mathfrak{Z(g)}$ and for all $x \in G$, we can see that
\begin{equation}\label{6.1.2}
X \cdot \phi (x)  =\sum\limits_{i=1}^{n} X \cdot\phi_i (x)\  v_i.
\end{equation}

 Therefore, we conclude that $X \cdot \phi_i = \chi(X) \phi_i$ for all $i$.\smallskip
 
 Each $\tau(k)$ has a matrix representation with respect to the chosen orthornormal basis of $V_{\tau}$. Let the $(i,j)$-th entry of $\tau(k)$ be denoted by $a_{ij}(k)$. \smallskip

 Note that $\phi(xk) =\tau(k^{-1})(\phi(x))$ for all $x \in G$ and $k \in K$.Thus we have
 
 \begin{equation}\label{6.1.1}
 \begin{split}
\sum_i \phi_i(xk) \ v_i & = \tau(k^{-1})(\sum_j \phi_j(x)\ v_j)\\
% & = \sum_j \phi_j(x) \tau(k^{-1})\ v_j\\
 % & = \sum_j \phi_j(x) \sum_l (\tau(k^{-1}))_{lj} \ v_l\\
& = \sum_l \left(\sum_j a_{lj}(k^{-1})\phi_j(x) \right)v_l.\\
 \end{split}
 \end{equation}

Therefore, $\phi_i(xk)=\sum\limits_{j=1}^{n} a_{ij}(k^{-1}) \ \phi_j(x)$\ for all $x \in G, \
 k \in K$.

%\begin{equation}
%\begin{split}
%X.\phi (x) & = \lim_{t \to 0} \frac{\phi(xe^{tX})-\phi(x)}{t}\\
%& = \lim_{t \to 0} \frac{\sum_i \phi_i(x e^{tX})v_i - \sum_i %\phi_i(x) v_i}{t}\\
%& = \sum\limits_{i=1}^{n} \lim_{t \to 0} %\frac{\phi_i(xe^{tX})-\phi_i(x)}{t} v_i\\
%& = \sum\limits_{i=1}^{n} X.\phi_i (x) v_i.\\
%\end{split}
%\end{equation}
%For $\phi \in V_{\chi, \Gamma, \tau}$, we have $X. \phi= \chi(X) %\phi$. \medskip

\subsection{Proof of Thm.~\ref{thm2}}

We denote $[\chi]=\{ \pi \in \widehat{G} : \
\text{the  infinitesimal character of} \ \pi \ \text{is}\ \chi\}$. Recall that $$ L^2(\Gamma \backslash G)= \widehat {\bigoplus\limits_{\pi \in \widehat{G}}} \ m(\pi, \Gamma) W_{\pi}.$$

 There are $m(\pi, \Gamma)$ many copies of $W_{\pi}$ say, $\{ W_{\pi_t} : 1 \leq t \leq m(\pi, \Gamma)\}$. Let $P_{\pi_t}$ be the projection onto $W_{\pi_t}$.\smallskip

Clearly, $\phi_i \in L^2(\Gamma \backslash G)$. We have
\[
\phi_i= \sum_{\pi \in [\chi]} \sum_{t=1}^{m(\pi, \Gamma)} P_{\pi_t} \phi_i.
\]

For any $1 \leq t \leq m(\pi, \Gamma)$, clearly $P_{\pi_t} \in \Hom_G(L^2(\Gamma \backslash G), W_{\pi_t})$. Using Eq.~\eqref{6.1.1}, we get

\begin{equation}\label{6.2.1}
\begin{split}
\pi(k) P_{\pi_t} \phi_i & = P_{\pi_t} \rho(k) \phi_i\\
& = P_{\pi_t} \sum_{j=1}^{n} a_{ij}(k^{-1})\phi_j\\
& = \sum_{j=1}^{n} a_{ij}(k^{-1}) P_{\pi_t} \phi_j.\\
\end{split}
\end{equation}

Let $F_t:=$ Span of $\{ P_{\pi_t} \phi_j : 1 \leq j \leq n\}$. From \eqref{6.2.1}, it follows that  $F_t$ is a finite dimensional representation of $K$. Let $\tau^\vee$ be the dual representation of $\tau$ of $K$ on the dual space $V_{\tau}^*$. Let $\{v_i^\ast\}_{i=1}^n$ be the dual basis of $V_{\tau}^*$. Then 
there is $K$-isomorphism between $F_t$ and $V_{\tau}^*$ which maps each $P_{\pi_t} \phi_i$ to $v_i^*$.\smallskip

Now we assume $\tau$ is irreducible (and hence $\tau^\vee$ as well). For a fixed $t$, let $W_{\pi_t}(\tau^\vee)$ be the isotypic component of $\tau^\vee$ in $W_{\pi_t}$. Then $F_t \subset W_{\pi_t}(\tau^\vee)$. We have a decomposition 
\[
W_{\pi_t}(\tau^\vee)=\bigoplus_{p=1}^{M_{\tau^\vee}} H_p,
\] 
where, each $H_p$ is an irreducible representation of $K$ isomorphic to $\tau^\vee$. Note that $M_{\tau^\vee} = \text{dim}(\Hom_K(\tau^\vee, \pi))$. \smallskip

For each $ p$, let $\{ \phi_{t, p, s}\}_{s=1}^{n}$ be a basis of $H_p$ satisfying $\phi_{p, t, s}(xk)=\sum\limits_{u=1}^{n} a_{us}(k^{-1}) \ \phi_{t, p, u}(x)$. Note that this is possible because each $H_p$ is isomorphic to $F_t$.\smallskip

 Now, $P_{\pi_t} \phi_i = \sum\limits_p \sum\limits_s \alpha_{p,s,i}\  \phi_{t, p, s}$. For each $p$, the matrix $(\alpha_{p,s,i})_{s i}$ is matrix which represents a $K$-homomorphism between $F_t$ and $H_p$.  Therefore, $\alpha_{p, s, i} = \alpha_p \delta_{s, i}$ for some scalar $\alpha_p$ by Schur's lemma.\smallskip

 Therefore, $P_{\pi_t} \phi_i = \sum\limits_p \alpha_p\ \phi_{t, p, i}$. We define $\phi_{t, p}(x):= \sum\limits_{i=1}^{n} \phi_{t, p, i}(x)\ v_i$ for all $t$ and $p$. Hence, we get
 \begin{equation}
 \begin{split}
 \phi & = \sum\limits_{i=1}^{n} \phi_i v_i \\
 & = \sum_{i=1}^{n} \left(\sum_{\pi \in [\chi]} \sum_{t =1}^{m(\pi, \Gamma)} P_{\pi_t}\ \phi_i\right)\ v_i\\
 & = \sum_{i=1}^{n} \sum_{\pi \in [\chi]} \sum_{t =1}^{m(\pi, \Gamma)} \sum_{p=1}^{M_{\tau^\vee}} \alpha_p\ \phi_{t, p, i}\ v_i\\
 & = \sum\limits_{\pi \in [\chi]} \sum_{t=1}^{m(\pi, \Gamma)} \ \sum_{p=1}^{M_{\tau^\vee}} \alpha_p\ \phi_{t, p}.
  \end{split}
 \end{equation}

We conclude that  dim $V_{\chi, \Gamma, \tau} \leq \sum \limits_{\pi \in [\chi]}m(\pi, \Gamma) \ { \rm dim}(\Hom_K(\tau^\vee, \pi))$.\smallskip

Conversely, for every $t$ and for every $p$, choose a basis $\{ \phi_{t, p, i}\}_{i=1}^{n}$ of $H_p$ such that 
\[\phi_{t, p, i}(xk)= \sum\limits_{j=1}^{n} a_{ji}(k^{-1}) \ \phi_{t, p, j}(x).\]

 Let $\phi_{t, p}= \sum\limits_{i=1}^{n}\phi_{t,p,i} \ v_i$. Then $\phi_{t, p}(xk)=\tau(k^{-1})(\phi_{t,p}(x))$. Since $X \cdot \phi_{t, p, i}= \chi(X)\ \phi_{t, p, i}$ for all $X \in \mathfrak{Z(g)}$, it follows that  $\phi_{t, p} \in V_{\chi, \Gamma, \tau}$. \smallskip

We conclude that $\sum\limits_{\pi \in [\chi]} m(\pi, \Gamma) \ {\rm dim}(\Hom_K(\tau^\vee, \pi)) \leq {\rm dim} \ V_{\chi, \Gamma, \tau}$.\smallskip

Therefore, we have the equality 
\begin{equation} \label{MM}
{\rm dim} \ V_{\chi, \Gamma, \tau} = \sum_{\pi \in [\chi]}m(\pi, \Gamma)\ {\rm dim}(\Hom_K(\tau^\vee, \pi))\ \text{ for all}\ \tau \in \widehat{K}.
\end{equation}

Now, let us consider the general case that 
$\tau$ is a finite dimensional representation of $K$ (possibly reducible). Let $\tau \cong \bigoplus \limits_{i=1}^{q}
m_i \ \tau_i$ be a decomposition of $\tau$ into irreducible representations of $K$. Hence $\tau^\vee \cong \bigoplus \limits_{i=1}^{q}
m_i \ (\tau_i)^\vee$. \smallskip

Let $V_{\Gamma, \tau}$ be as in Sec.\ \ref{iso-rep}. The center $\mathfrak{Z(g)}$ acts on this space. Therefore for any character $\chi \in \widehat{\mathfrak{Z(g)}}$, we can consider the $\chi$-eigenspace $V_{\chi, \Gamma, \tau} \subset V_{\Gamma, \tau}$. We observe that
\[ {\rm dim} \ V_{\chi, \Gamma, \tau}=\sum_{i=1}^{q}\ m_i\ {\rm dim}\ V_{ \chi, \Gamma, \tau_i}.\] 

We can use the argument culminating into Eq.\eqref{MM} for each $\tau_i$ and the additivity properties of $\Hom$ spaces with respect to decomposition of $\tau$ to complete the proof of Thm.\ \ref{thm2}. 
\smallskip

 \section{Discrete series representations and cohomology}\label{DS}
In this section, we comment about the case when $G$ is non-compact connected semisimple group that admits a discrete series representation. (This is same as saying rank $G$= rank $T$, where $T$ is a `{\it compact}' Cartan subgroup.) It is well-known that for a given $\chi \in \widehat{\mathfrak{Z(g)}}$, and a $K$-type $\tau$  there exists at most one (up to infinitesimal equivalence) discrete series representation $\pi_{\lambda + \rho}$ with the infinitesimal character $\chi$ and minimal $K$-type $\tau$ % depends on \lambda %
where $\lambda \in \mathfrak{t}^*_{\C}$ such that $\lambda + \rho $ is regular and integral linear form. Also, $\Hom_K(\tau, \pi_{\lambda + \rho})=1$.\smallskip

Therefore  Thm.~ \ref{thm2} implies the dimension of $V_{\chi, \Gamma, \tau^\vee}$ is equal to the multiplicity of the above discrete series $\pi_{\lambda+\rho}$ in $L^2(\Gamma \backslash G)$.\medskip

 For $\lambda \in \mathfrak{t}^*_{\C}$, there exists a $G$-equivariant holomorphic line bundle $\mathcal{L}_{\lambda}$ on $G/T$. Under a mild condition on $\lambda$ (see 7.66 \cite{MR1104438}), the theorem [Thm. 7.65, \cite{MR1104438}] implies the $L^2$-cohomology $H^*_2(\Gamma \backslash \mathcal{L}_{\lambda})$ is non-vanishing in exactly one degree $q_{\lambda}$ (depends on $\lambda$) and $$ {\rm dim}\ H_2^{q_\lambda}(\Gamma \backslash \mathcal{L}_{\lambda})=m(\pi_{\lambda +\rho}, \Gamma).$$

Hence, the dimension of the $\chi$-eigenspace of the automorphic forms of type $\tau^\vee$
is equal to the $q_{\lambda}$-th $L^2$-cohomology of the automorphic line bundle $\Gamma \backslash \mathcal{L}_{\lambda}$.

\end{document}